\documentclass[11pt,a4paper,reqno]{amsproc}
\usepackage[a4paper, margin=0.8in]{geometry}
\usepackage[utf8]{inputenc}
\usepackage{cite}
\usepackage{xcolor}
\usepackage{color}
\usepackage{calligra}
\usepackage[T1]{fontenc}
\usepackage{latexsym}
\usepackage[all]{xy}
\usepackage{hyperref}
\usepackage{mathrsfs}
\usepackage{graphicx}
\usepackage{amsmath,amsfonts,amssymb,amsxtra}
\numberwithin{equation}{section}
\newtheorem{theorem}{\bf Theorem}[section]

\newtheorem{definition}[theorem]{\bf Definition}

\newtheorem{example}[theorem]{\bf Example}

\newtheorem{lemma}[theorem]{\bf Lemma}
\newtheorem{remark}[theorem]{\bf Remark}

\newtheorem{obs}[theorem]{\bf Remark}

\newcommand{\R}{\mathbb{R}}
\newcommand{\E}{\mathbb{E}}

\newcommand{\qe}{\mathbb {Q}^{n+1}_{\varepsilon}}

\usepackage{algorithmic}
\begin{document}

\title[On the inverse mean curvature flow by parallel hypersurfaces in space forms]{On the inverse mean curvature flow by parallel hypersurfaces in space forms}

\author[Alancoc dos Santos Alencar]{Alancoc dos Santos Alencar$^1$}
\author[Keti Tenenblat]{Keti Tenenblat$^2$}
\address{$^1$ Departamento de Matemática, Universidade de Brasília, Campus Universitário Darcy Ribeiro, Brasília - DF, 70910-900, Brazil}
\address{$^2$ Departamento de Matemática, Universidade de Brasília, Campus Universitário Darcy Ribeiro, Brasília - DF, 70910-900, Brazil}
\thanks{$^1$ Supported by CAPES/Brazil - Finance Code 001}
\thanks{$^2$Partially supported by CNPq Proc. 311916/2021-0, Ministry of Science and Technology, Brazil and CAPES/Brazil - Finance Code 001.}

\begin{abstract}  
We consider the inverse mean curvature flow by parallel hypersurfaces in space forms. We show that such a flow exists if and only if the initial  hypersurface is isoparametric. The flow is characterized by 
 an algebraic equation satisfied by the distance function of the parallel hypersurfaces. The solutions to the flow are obtained explicitly
 when the distinct principal curvatures  have the same multiplicity.  This is an additional assumption only for  isoparametric hypersurfaces of the hyperbolic space or of the sphere with two or four distinct principal curvatures.  The boundaries of the maximal interval of definition, 
 when finite, are determined  in terms of the number $g$ of distinct principal curvatures, their multiplicities $m$  and  the mean curvature $H$ of  the initial hypersurface. We describe the collapsing submanifolds of the flow at the boundaries of the interval.
   In particular, we show in the Euclidean space  the solutions are eternal,  while in the hyperbolic space there are eternal and immortal solutions.  
  Starting with a connected  isoparametric submanifold of the sphere, we show that the flow is an ancient solution, that collapses into a minimal hypersurface whose square length of its second fundamental form and its scalar curvature are constants given in terms of $g$ and $n$.
The minimal hypersurface is totally geodesic when $g=1$, it is a Clifford minimal hypersurface of the sphere when $g=2$ and it is a Cartan type minimal submanifold when $g\in\{3,4,6\}$.  
   \end{abstract} 
\maketitle


\section{Introduction}\label{Introduction}
Let $\overline{M}^{n+1}$ be a Riemannian manifold with a given metric. Consider $F:M^{n} \longrightarrow \overline{M}^{n+1}$ an oriented connected  hypersurface with a unit normal vector field $N$. A one parameter family of hypersurfaces ${F}^t:M \longrightarrow \overline{M}$,  $t\in I$, $0 \in I \subset \R$, is a solution to the {\it {Inverse Mean Curvature Flow}} (IMCF), with initial data $F$, if
\begin{eqnarray}\label{eqIMCF}
\frac{\partial {F}^t}{\partial t}(x)&=&-\frac{1}{{H}^t(x)}{N}^t(x)\qquad \mbox{ for } (x,t)\in M \times I,  \\
{F}^0(x)&=&F(x)
\nonumber
\end{eqnarray}
where ${H}^{t}(.)$ is the mean curvature and ${N}^{t}(.)$ 
is the inward unit normal vector field of ${F}^{t}(.)$, with ${H}^{t}(x)>0$, $\forall (x,t) \in M\times I$. We say that ${F}^t$ is an {\it {ancient solution}} if $I=(-\infty,t^{*})$, with $t^{*}>0$. If $I=(t^{*},+\infty)$, with $t^{*}<0$, then ${F}^t$ is called an {\it {immortal solution}} and if $I=\R$ then  ${F}^t$ is called an {\it {eternal solution}}.

The inverse mean curvature flow has been studied extensively in recent decades. We mention two important papers related to the IMCF. In $2001$, using the IMCF, Huisken and Ilmanen \cite{HuiskenIlmanen} proved the Riemannian Penrose Inequality. The Penrose inequality was conjectured in $1973$ and it relates, under some conditions, the mass of certain spacetimes to the area of their black holes. The semi-Riemannian case remains an open problem.  On the other hand, in $2004$, using IMCF, Bray and Neves \cite{BrayNeves} proved the Poincaré conjecture for $3$-manifolds with Yamabe invariant greater than the Yamabe invariant of real projective $3$-space. There are several papers that relate the IMCF with physical problems involving general relativity and black holes, see for example: \cite{Bray}, \cite{LeeandNeves}, \cite{Ger} and \cite{Ger2}. 

\newpage 

The mean curvature flow by parallel hypersurfaces in simply connected space forms was studied by Reis and Tenenblat \cite{ReisTenenblat}, where it is proved that such a flow exists if and only if  
the initial hypersurface is isoparametric. Moreover, by solving an ordinary differential equation, explicit solutions are given for all isoparametric hypersurfaces in space forms. In particular, for such hypersurfaces of the sphere, the exact  collapsing time is given in terms of its dimension, the principal curvatures and their multiplicities. In \cite{GuiSantosSantos}, the authors considered the evolution of the mean curvature flows in  Riemannian manifolds, by isoparametric hypersurfaces.

 The theory  of the mean curvature flow was extended  by  de Lima \cite{deLima} to Weingarten flows by parallel hypersurfaces. For an orientable hypersurface  
$f:M^n\rightarrow \bar{M}^{n+1}$, with principal curvature functions $k_1,...,k_n$ in a normal direction $N$,   a Weingarten function of $f$, $W_f(k_1,...k_n)$ is defined as a function which is symmetric, homogeneous, monotone increasing with respect to any of its variables and positive when $k_i>0,\, \forall i$. The author considers a Weingarten flow, i.e.,  a 1-parameter family of smooth oriented hypersurfaces $F_t: M^n\rightarrow \bar{M}^{n+1}$ satisfying the evolution equation $\partial F_t/\partial t=W_{F_t} N$, $t\in [0,T)$, with $F_0=f$ and  seeks for solutions $F_t$ that are parallel to the initial data $f$. 
 Without proving that for Weingarten flows the initial hypersurface must be necessarily isoparametric, but  motivated by \cite{ReisTenenblat} and by the existence  of isoparametric hypersurfaces in simply connected space forms as well as in rank one symmetric spaces of non compact type, the  author assumes that the initial immersion $f$ is isoparametric. Starting with such a hypersurface, existence results are obtained when the ambient space $\bar{M}^{n+1}$ is either a simply connected space form or a rank-one 
 symmetric space of noncompact type. Whenever the ambient space is the sphere some additional  conditions are required for the Weingarten function. 
Moreover, in \cite{deLima},  for complete connected hypersurfaces,  an embedding preserving property and an avoidance principle are proved.

In this paper, motivated by the importance of the inverse mean curvature flow, we study in detail the IMCF  \eqref{eqIMCF}, by parallel hypersurfaces in simply connected space forms. This flow is in the expanding direction as considered by Huisken - Ilmanen \cite{HuiskenIlmanen}, i.e., the mean curvature is positive (the normal vector field is inward) and the flow is in the outward normal direction. We observe that $-1/H$ does not satisfy the requirement of a  Weingarten function, as defined in \cite{deLima}, since it is not positive whenever  all the principal curvatures are positive. Therefore, we point out that  none of our results is contained in \cite{deLima}.

 In analogy to \cite{ReisTenenblat}, we start showing that a family of parallel hypersurfaces evolves under the IMCF if, and only if, the initial hypersurface is isoparametric (Theorem \ref{res.16}). In this theorem, we characterize such flows in terms of an algebraic equation satisfied by the distance function of the parallel hypersurfaces.   
 Using Theorem \ref{res.16} and the complete characterization of isoparametric hypersurfaces by Segre \cite{Segre} in Euclidean space, by Cartan \cite{Cartan} in the hyperbolic space and by work of Cartan \cite{Cartan} and M\"unzner \cite{Munzner} in the sphere,  
  we obtain  explicitly the solutions  to the IMCF, when the distinct principal curvatures  have the same multiplicity $m$. 
We point out that this is an additional assumption only for  
isoparametric hypersurfaces of the hyperbolic space or of the sphere
with $2$ or $4$ distinct principal curvatures. 
We consider the maximal interval of definition of the flow in the whole $\mathbb{R}$, which allows us to determine if the solutions are eternal, ancient or immortal. The finite boundaries of  the maximal interval of definition  of the flow are given explicitly in terms of the number $g$ of 
 distinct principal curvatures, their multiplicities $m$ and the mean curvature of  the initial hypersurface.   We also  show that when the flow tends to the boundary of the interval of definition, it converges to certain submanifolds of the space form.

Theorem \ref{res.12novo} provides the results in the Euclidean space. In Theorems \ref{res.11}-\ref{res.17},  we treat the flow in the hyperbolic space and in Theorem \ref{resgeral} we provide the results for hypersurfaces in the sphere. 
   In particular, starting with an isoparametric submanifold of the sphere,  we show that the flow is an ancient solution, i.e.,  it is defined in the interval $(-\infty, t^*)$, where 
 $t^*=m/2\ln (H^2/n^2+1)$. The flow expands from  
  an $m(g-1)$-dimensional submanifold and it collapses into a minimal hypersurface (when $t\rightarrow t^*$),  which is a totally geodesic hypersurface of the sphere if $g=1$, 
it is a Clifford minimal hypersurface of the sphere, when $g=2$ and it is a Cartan type minimal hypersurface \cite{Cartan}, when $g\in\{3,4,6\}$. 
  The  square length of its second fundamental form $|A|^2$  and its scalar curvatures $R$ are constants 
given by $|A|^2=n(g-1)$ and $R=n(n-g)$ respectively  (see Theorem \ref{resgeral}).     

 We should mention that the additional assumption mentioned above, in the case of 2 or 4 distinct principal curvatures, is due to the difficulty of solving the algebraic equation of the distance function  whenever the distinct principal curvatures do not have the same multiplicity. It would be interesting to investigate the IMCF in this case. 
 
 This paper is organized as follows. 
  In Section \ref{preliminaries}, we fix our notation and include some basic results that will be used in the following sections. 
    In Section \ref{Main results}, we state our main results, 
   whose proofs are given in Section \ref{Proofs}.
 In Section \ref{visual}, we visualize some solutions to the IMCF.

\section{Preliminaries}\label{preliminaries}
Let $\overline{M}^{n+1}$ be  a simply connected space form, i.e., $\overline{M}^{n+1}$ is the Euclidean space $\mathbb{E}^{n+1}$, the sphere $\mathbb{S}^{n+1} \subset \E^{n+2}$ or the hyperbolic space $\mathbb{H}^{n+1} \subset \mathbb{L}^{n+2}$, where $\mathbb{L}^{n+2}$ is the Lorentz space of dimension $n+2$. 
From now on, we will consider  connected  oriented hypersurfaces  of  $\overline{M}^{n+1}$, whose mean curvature $H$ does not vanish, therefore w.l.o.g. we will chose the unit normal vector such that  $H>0$.  
In what follows, we will denote 
the space form with the standard metric $\langle \cdot, \cdot\rangle$  by $\mathbb {Q}^{n+1}_{\varepsilon}$, where
\begin{eqnarray}
\mathbb {Q}^{n+1}_{\varepsilon}= \left\{ \begin{array}{ll}
\mathbb {S}^{n+1},	& \mbox{if} ~~ \varepsilon = 1, \nonumber \\ 
\mathbb {E}^{n+1},	& \mbox{if} ~~ \varepsilon = 0, \nonumber \\
\mathbb {H}^{n+1},	& \mbox{if} ~~ \varepsilon = -1, \nonumber
\end{array} 
\right.
\end{eqnarray}
 and $
\mathbb {H}^{n+1}=\{(x_{1},\ldots,x_{n+1}) \in \mathbb {L}^{n+2}:-x_{1}^{2}+   
\sum_{i=2}^{n+2}x_i^2=-1 \mbox{, } x_{1}>0\} 
$. 

Consider $F:M^{n} \longrightarrow \qe$ an oriented hypersurface with the metric induced by  the space form. 
The quadratic form defined by metric is  the {\it {first fundamental form}} of the hypersurface $F(M)$ in $x \in M$, $I_x(v,w)=\langle dF_{x}(v),dF_{x}(w) \rangle$, $u,v \in T_{x}M$ . 
If $N$ is a unit normal vector field, the {\it{second fundamental form}} of $M$ is given by 
$ II_{x}(v,w)=-\langle dN_{x}(v),w \rangle$, 
$\forall v, w \in T_{x}M$, $x \in M$. Let $E_{1},\ldots, E_{n}$ be orthonormal vector fields which are {\it{principal directions}} and let $k_{1},\ldots, k_{n}$ be the {\it{principal curvatures}} of $M$, i.e.,
$g_{x}(E_{i},E_{j})=\delta_{ij}$   and   $II_{x}(E_{i},E_{j})=k_{i}(x)\delta_{ij}$, 
$x \in M$, $1\leq i \leq n$ and $1\leq j \leq n$. We will denote the {\it{mean curvature}} by $H(x)=\sum_{i=1}^{n}k_{i}(x)$.
When the principal curvatures $k_{i}$ of $M$ are constant, for all $i=1,...,n$,  $M$ is an {\it{isoparametric hypersurface}}.

\begin{definition}\label{res.3}{\rm
Let $F: M^{n} \longrightarrow \qe$ be an oriented hypersurface with unit normal vector field $N$. The {\em parallel hypersurface} to $F(M)$ at a distance $\mu \in \R$ (with sign) is the map $ {F}^{\mu}: M^{n} \longrightarrow \qe$ such that, for each $x \in M^{n}$, 
$F^\mu (x)$ is the point at distance $\mu$ along the geodesic in $\qe$  starting 
at $F(x)$, with tangent vector $N(F(x))$.  
} 
\end{definition}

Considering the following notation
\begin{eqnarray}\label{eq.p3}
C_{\varepsilon}(\mu)= \left\{\begin{array}{ll}
\cos \mu,	& \mbox{if} ~~ \varepsilon = 1, \\ 
1,	& \mbox{if} ~~ \varepsilon = 0, \\
\cosh \mu,	& \mbox{if} ~~ \varepsilon = -1,
\end{array} 
\right. 
\mbox{ and ~~} 
S_{\varepsilon}(\mu)= \left\{\begin{array}{ll}
\sin \mu,	& \mbox{if} ~~ \varepsilon = 1, \\ 
\mu,	& \mbox{if} ~~ \varepsilon = 0, \\
\sinh \mu,	& \mbox{if} ~~ \varepsilon = -1, 
\end{array} 
\right.  
\end{eqnarray}
 the parallel hypersurface $F^\mu(M^{n})$ in $\qe$ are given by 
\begin{eqnarray}\label{eq.p4}
F^\mu (x)=	C_{\varepsilon}(\mu)F(x)+S_{\varepsilon}(\mu)N(F(x)),
\end{eqnarray}
where $\mu \in \R$ and $x \in M^{n}$. Now, we recall some basic facts about the parallel hypersurface $F^\mu(M^{n})$.

\begin{lemma}\label{res.4}
\noindent $(i)$ The unit normal vector field $ {N}^{\mu}$ of $F^\mu (M)$ is given by
\[
 {N}^{\mu}(x)=-\varepsilon S_\varepsilon(\mu)F(x)+C_\varepsilon(\mu)N(F(x)). \]

\noindent $(ii)$ If $E_{1},\ldots, E_{n} \subset T_{x}M$ is an orthonormal basis of eigenvectors of the second fundamental form $II_{x}$ of $F(M)$ in $x \in M$ and $k_{1},\ldots, k_{n}$ are the principal curvatures of $F(M)$ in $x \in M$ with respect to $N$, then the first fundamental form is given by 
\[
{g}^{\mu}_{x}(E_{i},E_{j})=\left[C_\varepsilon(\mu)-k_{i}(x) S_\varepsilon(\mu)\right]^{2} \delta_{ij}.  \nonumber
\]

\noindent $(iii)$ The second fundamental form ${II}^{\mu}_{x}$ of $F^\mu$ at $x \in M$ is given by
\[
{II}^{\mu}_{x}(E_{i},E_{j})=\left[C_\varepsilon(\mu)-k_{i}(x) S_\varepsilon(\mu)\right] \left[\varepsilon S_\varepsilon(\mu)+k_{i}(x) C_\varepsilon(\mu)\right] \delta_{ij}. 
\]

\noindent $(iv)$  The principal curvatures ${k}_{i}^{\mu}(x)$, $i=1$, $2$,..., $n$, of $F^\mu$ at $x \in M$, with respect to ${N}^{\mu}$, are given by
\[
{k}_{i}^{\mu}(x)=\frac{\varepsilon S_{\varepsilon}(\mu)+k_{i}(x)C_{\varepsilon}(\mu)}{C_{\varepsilon}(\mu)-k_{i}(x)S_{\varepsilon}(\mu)}.
 \]
\end{lemma}

\begin{lemma}[Cartan \cite{Cartan}]\label{res.5}
Let ${F}^{\mu}: M^{n} \longrightarrow \qe$, $\mu \in (-\delta, \delta)$ with $\delta>0$, a family of parallel hypersurfaces. Then, $F(M)$ is isoparametric if and only if the mean curvature ${H}^{\mu}(x)$ of $ {F}^{\mu}(M)$ does not depend on $x \in M$, $\forall \mu \in (-\delta, \delta)$.
\end{lemma}

We will now consider a special class of solution to the IMCF: namely the solutions given by a family of parallel hypersurfaces in space forms. From now on, we focus on this type of solutions.

\begin{definition} {\rm
Let ${F}^t:M^{n}  \longrightarrow \qe$, $t\in I$,  be a solution to the IMCF in $\qe$ with initial data $F:M^{n} \longrightarrow \qe$. We say that ${F}^t$ is a {\em solution to the IMCF by parallel hypersurfaces} if there is a function $\mu: I \longrightarrow \R$ such that
\begin{eqnarray}\label{eq.p12}
{F}^{t}(x)= C_{\varepsilon}(\mu(t))F(x)+S_{\varepsilon}(\mu(t))N(x),
\end{eqnarray}
$\forall t \in I$, where $C_{\varepsilon}$ and $S_\varepsilon$ are the functions defined in $(\ref{eq.p3})$.}
\end{definition}

\section{Main results}\label{Main results}

In this section we state our main results. The proofs will be given in Section \ref{Proofs}.

\begin{theorem}\label{res.16}
Let $F:M^{n} \longrightarrow \qe$ be a connected  oriented hypersurface with inward   unit normal vector field $N$, i.e., such that the mean curvature  satisfies $H(x)>0$, $\forall x \in M$. Then, $F(M)$ is an initial data of a solution to the IMCF by parallel hypersurfaces if and only if $F(M)$ is an isoparametric hypersurface.
Moreover, for such a hypersurface,   
 the solution to the IMCF by parallel hypersurfaces with initial data $F(M)$ is given by
\begin{eqnarray}\label{eq.p.113}
 {F}^{t}(x)= C_{\varepsilon}(\mu(t))F(x)+S_{\varepsilon}(\mu(t))N(x),
\end{eqnarray}
$\forall (x,t) \in M \times I$, where $0 \in I \subset \R$, $C_{\varepsilon}$ and $S_{\varepsilon}$ are the functions given by $(\ref{eq.p3})$ and $\mu : I \longrightarrow \R$ is a smooth real function satisfying $\mu(0)=0$ and
\begin{eqnarray}\label{eq.p.114}
\prod_{i=1}^{n}\left( \, C_{\varepsilon}(\mu(t))-k_{i}S_{\varepsilon}(\mu(t))\, \right) = e^{t}, ~~ \forall t \in I, 
\end{eqnarray}
where $k_{i},\, i=1,...,n$ are the principal curvatures of $F(M)$, associated to $N$. 
\end{theorem}

From now on, using the well known properties of the isoparametric hypersurfaces in space forms and Theorem $\ref{res.16}$, we will describe in detail the solutions to the IMCF by parallel hypersurfaces and their behavior on the maximal interval where they are defined.  

We start with the hypersurfaces of the Euclidean space. Let  $F:M^{n} \longrightarrow \E^{n+1}$ be an isoparametric hypersurface in the Euclidean space, then Segre \cite{Segre} proved that $M$ is isometric to one of the three following hypersurfaces: 
i) $\E^{n}$, with all the principal curvatures equal to $0$;
 ii) $\mathbb{S}^{n}_{r_0}$, an $n$-dimensional sphere of radius $r_0$,  with all the principal curvatures equal to $k=1/r_0$;
 iii) $\mathbb{S}^{m}_{r_0} \times \E^{n-m}$, a cylinder with $m$ principal curvatures equal to $k_{1}=1/r_0$ and $n-m$ principal curvatures equal to $0$.

Using Theorem $\ref{res.16}$, our  theorem  describes the IMCF for the spheres and for the cylinders in the euclidean space. 

\begin{theorem}\label{res.12novo}
	Let $F: \mathbb {S}^{m}_{r_0} \times \E^{n-m} \longrightarrow \E^{n+1}$, $m \neq 0$,  be an isometric immersion of a sphere, when $m=n$ or a cylinder when $m\neq n$, in the Euclidean space $\E^{n+1}$, with inward unit normal vector field $N$. Let $m$ principal curvatures be equal to $k=1/r_0$ and $n-m$ null principal curvatures, when $m\neq n$. 
 Then, the solution to the IMCF by parallel hypersurfaces with initial condition $F$ is  eternal and it is given by \\ 
 (i) If $m=n$,  
 \begin{eqnarray}
{F}^{t}(x)=e^{t/n}F(x), \nonumber, \qquad \forall  t \in \R, \mbox{ and } x \in \mathbb {S}^{n}_{r_{0}}.  
\end{eqnarray}
Moreover,  ${F}^{t}(\mathbb{S}^{n}_{r_{0}})$ is a sphere centered at the origin with radius $r(t)=r_{0}e^{t/n}$ that converges to a point, when $t\rightarrow 
 -\infty$.  \\ 
(ii) If $m\neq n$ 
 \begin{eqnarray}
	{F}^{t}(x,y)=F(x,y)+r_0 (1-e^{t/m})N(x,y), \nonumber \qquad \forall t \in \R,  (x,y) \in \mathbb {S}^{m}_{r_0} \times \E^{n-m}.
	\end{eqnarray}
Moreover, 	${F}^{t}(\mathbb{S}^{m}_{r_0} \times \E^{n-m})$    is a cylinder whose base has radius $r(t)=r_{0}e^{t/m}$ and it converges to an $(n-m)$-dimensional Euclidean subspace of $\E^{n+1}$, when  $t\rightarrow -\infty$.   
\end{theorem}

Now we consider the isoparametric hypersurfaces  of the hyperbolic space 
$\mathbb{H}^{n+1}$.   Cartan \cite{Cartan} proved that any such hypersurface  is isometric to one of the following hypersurfaces:
i)
$\mathbb{H}^{n}$, with all the principal curvatures equal to $k=0$;
 ii) $\E^{n}$, with all the principal curvatures equal to $k=\pm1$;
   iii) $\mathbb{H}^{n}$, with all the principal curvatures equal to $k$, where $0<|k|<1$;
  iv) $\mathbb{S}^{n}$, with all the principal curvatures equal to $k$, where $|k|>1$;
v) $\mathbb{S}^{m} \times \mathbb{H}^{n-m}$, with $m$ principal curvatures equal to $k_{1}$ and $n-m$ principal curvatures equal to $k_{2}$, where $k_{1}k_{2}=1$.  The  description of the IMCF of the isoparametric hypersurfaces of $\mathbb{H}^{n+1}$ of type ii), iii) and iv) are given in Theorems $\ref{res.11}$ and  $\ref{res.22}$ below. In Theorem $\ref{res.17}$, we consider the IMCF in case v), with the additional assumption that the two distinct principal curvatures have the same multiplicity $m$, i.e., $n=2m$.
In these theorems, the IMCF is defined on a maximal interval $I=(t_0, +\infty)$, where 
$t_0\geq -\infty$. In order to describe the limiting behavior of the flow when 
$t\rightarrow +\infty$, we consider the conformal isometry $G$ between $\mathbb{H}^{n+1}\subset \mathbb{L}^{n+2}$ and the Poincar\'e ball model 
of the hyperbolic space $\mathbb{B}^{n+1}\subset \mathbb{E}^{n+1}$ (see proof of Lemma 
\ref{res.23}) and we define the conformal boundary of the hyperbolic space,  $\partial \mathbb{H}^{n+1}$, as the inverse image of the boundary of the Poincar\'e ball, i.e.,  
$\partial \mathbb{H}^{n+1}=G^{-1}(\partial \mathbb{B}^{n+1})$.

\begin{theorem}\label{res.11}
Let $F: \E^{n} \longrightarrow \mathbb{H}^{n+1} \subset \mathbb{L}^{n+2}$ be an isometric immersion of a horosphere in the hyperbolic space, with 
 unit normal vector field $N$ with positive mean curvature and all principal curvatures equal to $k=\pm 1$. Then, the solution to the IMCF, by parallel hypersurfaces with initial condition $F$, is given by 
\begin{eqnarray}\label{eq.p.293}
{F}^{t}(x)=\cosh\left(-k\frac{t}{n}\right)F(x)+\sinh\left(-k\frac{t}{n}\right)N(x), \nonumber
\end{eqnarray}
$\forall t \in \R$ and $x\in \E^n$, i.e., it  is an eternal solution.  Moreover, for each $t\in \R$, ${F}^{t}(\E^{n})$ is a horosphere,  that  converges to a point when $t\rightarrow -\infty$  and it converges to the conformal boundary of $\mathbb{H}^{n+1}$, $\partial\mathbb{H}^{n+1}$, when $t \rightarrow +\infty$.
\end{theorem}

\begin{theorem}\label{res.22}
Let $F: M^{n} \longrightarrow \mathbb{H}^{n+1} \subset \mathbb{L}^{n+2}$ be a  connected isometric immersion of a totally umbilic hypersurface in the hyperbolic space, with  inward 
unit normal vector field $N$ and  all the principal curvatures equal to $k$, where $k \notin \{-1,0,1\}$. Then, the solution to the IMCF by parallel hypersurfaces with initial condition $F$ is given by 
\begin{eqnarray}
{F}^{t}(x)=\cosh(\mu(t))F(x)+\sinh(\mu(t))N(x), \nonumber
\end{eqnarray}
 where
\[
\cosh(\mu(t))=\frac{e^{t/n}-|k|\sqrt{q(t)}}{1-k^{2}}, \qquad  \sinh(\mu(t))=sgn(k)\frac{|k|e^{t/n}-\sqrt{q(t)}}{1-k^{2}}, 
\]
$\forall t \in I$ and $\forall x \in M$, and 
$$q(t)=e^{2t/n}+k^{2}-1.$$
\item $(i)$ If $0<|k|<1$, then $I=(t^{*},+\infty)$, where $t^{*}=\frac{n}{2}\ln(1-k^{2})$, i.e.,  ${F}^{t}$ is an immortal solution. Moreover, ${F}^{t}(M)$ converges to a totally geodesic hypersurface of $\mathbb{H}^{n+1}$ when $t \rightarrow t^*$ and it converges to $\partial\mathbb{H}^{n+1}$ when 
$t \rightarrow +\infty$.

\item $(ii)$ If $|k|>1$, then $I=\R$, i.e., ${F}^{t}$ is an eternal solution. Moreover, ${F}^{t}(M)$ converges to a point when $t \rightarrow -\infty$ and it converges to $\partial\mathbb{H}^{n+1}$ when $t \rightarrow +\infty$.
\end{theorem}

\begin{theorem}\label{res.17}
	Let $F:\mathbb{S}^m \times \mathbb{H}^m \longrightarrow \mathbb{H}^{n+1} \subset \mathbb{L}^{n+2}$ be a connected  isometric immersion of a cylinder in the hyperbolic space, with 
	inward  
	unit normal vector field $N$ and two distinct principal curvatures, $k_{1}>1$ and $k_{2}$, with the same multiplicity $m$, such that $k_{1}k_{2}=1$. Then, the solution to the IMCF by parallel hypersurfaces, with initial condition $F$, is given by
	\begin{eqnarray}
	{F}^{t}(x)=\cosh(\mu(t))F(x)+\sinh(\mu(t))N(x), \nonumber
	\end{eqnarray}
	where
	\begin{eqnarray}\label{eq.p.104}
	\cosh (2\mu(t))=\frac{-e^{t/m}+a\sqrt{q(t)}} {a^2-1}, 
	\qquad 
	\sinh (2\mu(t))=\frac{-a e^{t/m}+\sqrt{q(t)}} {a^2-1}
	\end{eqnarray}
where
 $$q(t)= e^{2t/m}+ a^2-1, \quad \mbox{ and }\quad a=\frac{k_1+k_2}{2}= \frac{H}{n}$$ 
 $\forall t \in \R$, and $x\in \mathbb{S}^m \times \mathbb{H}^m$, i.e., ${F}^{t}$ is an eternal solution.
 Moreover, for each $t\in\R$,  ${F}^{t}(\mathbb{S}^{m}\times \mathbb{H}^{m})$ is a cylinder in the hyperbolic space with prinicpal curvatures satisfying $k_1^tk_2^t=1$, which converges to an $m$-dimensional submanifold of $\mathbb{H}^{n+1}$ when  
 when $t \rightarrow -\infty$ and it converges to  
  $\partial\mathbb{H}^{n+1}$ when $t \rightarrow +\infty$.
\end{theorem}

We now consider the isoparametric hypersurfaces of the sphere. M\"unzner \cite{Munzner} showed that the number $g$ of distinct principal curvatures, for an isoparametric hypersurface $F:M^{n} \longrightarrow \mathbb{S}^{n+1}$, is restricted to be $1$, $2$, $3$, $4$ or $6$. Cartan \cite{Cartan} classified these hypersurfaces when $g \leq 3$. If $g=1$, then $F(M)$ is a sphere obtained as the intersection of $\mathbb{S}^{n+1}$ with a hyperplane of $\E^{n+2}$. If $g=2$, then $F(M)$ must be the standard product of spheres $\mathbb{S}^{l}_{r_1} \times \mathbb{S}^{n-l}_{r_2} \subset \mathbb{S}^{n+1}$, where $r_{1}^{2}+r_{2}^{2}=1$. If $g=3$, Cartan proved that there are only four distinct isoparametric hypersurfaces of $\mathbb{S}^{n+1}$ with three distinct principal curvatures. Their dimensions are $n=3m$, where $m=1$, $2$, $4$ or $8$ and all the principal curvatures have the same multiplicity $m$. If $g=4$, M\"unzner \cite{Munzner}      proved that the principal curvatures $k_{1}$, $k_{2}$, $k_{3}$, $k_{4}$ can be ordered so that their corresponding multiplicities satisfy $m_1=m_3$ and $m_2=m_4$. If $g=6$, M\"unzner \cite{Munzner} showed that all the principal curvatures must have the same multiplicity $m$ and Abresch \cite{Abresch} showed that $m=1$ or $m=2$.

\begin{obs}\label{res.13}
{\normalfont {Let $F:M^{n} \longrightarrow \mathbb{S}^{n+1}$ be a connected  oriented isoparametric hypersurface with inward   
unit normal vector field $N$, i.e., $H>0$ and let $k_{1},\ldots,k_{n}$ be the principal curvatures of $F(M)$ associated to $N$. One can determine the principal curvatures of $M^{n} \subset \mathbb{S}^{n+1}$ up to a constant. In fact, for all $g\in \{1,2,3,4,6\}$, let $\tau \in \R$ with $-1<\tau<1$ and $\tau=\cos(g s)$, i.e., $0<s<\frac{\pi}{g}$, we consider 
	\begin{eqnarray}\label{eq.p.19}
	k_{j}=\cot\left(s+\frac{j-1}{g}\pi\right),~j=1,...,g. \nonumber
	\end{eqnarray}
	The hypersurfaces for different values of the constant $\tau$ are parallel in $\mathbb{S}^{n+1}$.
 }}
\end{obs}

We now describe the solutions to the IMCF by parallel hypersurfaces considering as initial conditions the isoparametric hypersurfaces of the sphere $\mathbb{S}^{n+1}$ with $g$ distinct principal curvatures, with the same multiplicity.   In Theorem \ref{resgeral},  this is an additional assumption only for  $g=2$  or $g=4$, since in these cases the principal curvatures may have different multiplicities. Whenever $g\in\{1,3,6\}$ the principal curvatures have the same multiplicities.    

\begin{theorem}\label{resgeral}
Let $F: M^{n} \longrightarrow \mathbb{S}^{n+1} \subset \E^{n+2}$ be a connected  isoparametric hypersurface in $\mathbb{S}^{n+1}$, with inward 
unit normal vector field $N$, i.e., $H>0$ and $g$ distinct principal curvatures $k_{i}$, $i=1,\dots,g$, where $g\in\{1,2,3,4,6\}$,  all with multiplicity $m$, i.e.,  $n=gm$.  
Then, the solution to the IMCF by parallel hypersurfaces with initial condition $F$ is given by
\begin{eqnarray}
{F}^{t}(x)=\cos(\mu(t))F(x)+\sin(\mu(t))N(x), \nonumber
\end{eqnarray}
$\forall x \in M$, where
\begin{eqnarray}
\cos(g\mu(t))=\frac{e^{t/m}+a\sqrt{q(t)}}{a^{2}+1},\qquad  \sin(g\mu(t))=\frac{-ae^{t/m}+\sqrt{q(t)}}{a^{2}+1}, \nonumber
\end{eqnarray}
\begin{eqnarray}
q(t)=a^{2}+1-e^{2t/m} \quad  \text{ and } \quad  a=  \frac{1}{g}\sum_{i=1}^g k_{i}=\frac{H}{n}, \nonumber
\end{eqnarray}
with $t \in (-\infty,t^{*})$ and $t^{*}=\frac{m}{2}\ln(a^{2}+1)$, i.e., $ {F}^{t}$ is an ancient solution.  Moreover, when $t\rightarrow -\infty$ 
 $ {F}^{t}(M)$ converges to an  $m(g-1)$-dimensional submanifold of $\mathbb{S}^{n+1}$   and when $t \longrightarrow t^{*}$ it converges to a minimal hypersurface $F^{t^*}(M)$  
of $\mathbb{S}^{n+1}$, whose  squared length  of its second fundamental form is   
$|A|^2= m g(g-1)=n(g-1)$ and its scalar  curvature is given by 
$R=m(m-1)g^2=n(n-g)$. 
In particular,  $F^{t^*}(M)$ is a totally geodesic submanifold of the sphere when $g=1$, it is a Clifford minimal hypersurface of the sphere, when $g=2$, and it is a Cartan type minimal submanifold when $g\in\{3,4,6\}$.       
\end{theorem}

\section{Proof of the main results}\label{Proofs}

\begin{proof}[{\bf{Proof of Theorem \ref{res.16}}}]
Consider a one parameter family of parallel hypersurfaces ${F}^t:M^{n} \longrightarrow \qe$, $t\in I\subset \R$ where $0 \in I$, given by $(\ref{eq.p12})$, i.e., 
\begin{eqnarray}\label{eq.p.106}
 {F}^{t}(x)= C_{\varepsilon}(\mu(t))F(x)+S_{\varepsilon}(\mu(t))N(x), 
 \quad \mbox{ and }\quad F^0(x)=F(x),
\end{eqnarray}
$\forall (x,t) \in M \times I$, where $C_{\varepsilon}$, $S_\varepsilon$ are the functions defined in \eqref{eq.p3} and $\mu: I \longrightarrow \R$ is a smooth real function such that ${F}^t$ is a solution to the IMCF with initial data $F(M)$. Then, 
\begin{eqnarray}\label{eq.p.108}
\frac{\partial  {F}^t}{\partial t}(x,t)=-\frac{1}{ {H}^t(x)}{N}^t(x),~\forall (x,t) \in M \times I,
\end{eqnarray}
where  ${H}^t(x)>0$, $\forall t \in I$, $\forall x \in M$.  It follows from $(\ref{eq.p.106})$ that $\mu(0)=0$. Now, substituting $(\ref{eq.p.106})$ into $(\ref{eq.p.108})$, we get, 
\begin{eqnarray}
\mu'(t)\left[-\varepsilon S_{\varepsilon}(\mu(t))F(x)+C_{\varepsilon}(\mu(t))N(x) \right]  =-\frac{1}{{H}^t(x)}{N}^t(x). \nonumber
\end{eqnarray}
Therefore, from Lemma \ref{res.4}, we have that $(\ref{eq.p.108})$ occurs if and only if
$ \mu'(t)=-\frac{1}{ {H}^t(x)}$, 
$\forall (x,t) \in M \times I$. Thus, the mean curvature ${H}^{t}$ of ${F}^{t}(M)$ does not depend on $x \in M$, $\forall t \in I$. It follows from Lemma $\ref{res.5}$ that $F(M)$ is an isoparametric hypersurface.

Conversely, let $F:M^{n} \longrightarrow \qe$ be an oriented isoparametric hypersurface with  unit normal vector field $N$, such that the mean curvature is positive. Let $\mu: I \longrightarrow \R$, $0 \in I \subset \R$, be the smooth real function that is the unique solution of the problem:
\begin{eqnarray}\label{eq.p.110}
\mu'(t)=-\frac{1}{\sum_{i=1}^{n}\frac{\varepsilon S_{\varepsilon}(\mu(t))+k_{i}C_{\varepsilon}(\mu(t))}{C_{\varepsilon}(\mu(t))-k_{i}S_{\varepsilon}(\mu(t))}} ~ \text{  and  } ~ \mu(0)=0, 
\end{eqnarray}
where $k_{i}$, $1 \leq i \leq n$, are the principal curvatures of $F(M)$ with respect to $N$ and $C_{\varepsilon}$, $S_{\varepsilon}$ are given by $(\ref{eq.p3})$. Consider the family of parallel hypersurfaces ${F}^t:M^{n}  \longrightarrow \qe$ given by \eqref{eq.p.106}.
We need to show that ${F}^{t}$ is a solution to the IMCF by parallel hypersurfaces with initial data $F$. Indeed, ${F}^0(x)=F(x)$ and we get from Lemma $\ref{res.4}$ that
\begin{eqnarray}\label{eq.p.111}
\frac{\partial  {F}^t}{\partial t}(x,t)=\mu'(t)\left[-\varepsilon S_{\varepsilon}(\mu(t))F(x)+C_{\varepsilon}(\mu(t))N(x)\right]=\mu'(t) {N}^t(x),
\end{eqnarray}
$\forall (x,t) \in M \times I$. Moreover, it follows from Lemma $\ref{res.4}$ and $(\ref{eq.p.110})$ that, $\forall (x,t) \in M \times I$,
\begin{eqnarray}\label{eq.p.112}
 {H}^t(x)=\sum_{i=1}^{n}{k}_{i}^t(x)=\sum_{i=1}^{n}\frac{\varepsilon S_{\varepsilon}(\mu(t))+k_{i}C_{\varepsilon}(\mu(t))}{C_{\varepsilon}(\mu(t))-k_{i}S_{\varepsilon}(\mu(t))}=-\frac{1}{\mu'(t)}.
\end{eqnarray}
Then, from $(\ref{eq.p.111})$ and $(\ref{eq.p.112})$, we have that, $\forall (x,t) \in M \times I$,
$\frac{\partial {F}^t}{\partial t}(x,t)=- {N}^t(x)/{H}^t(x)$, 
which concludes the first part of the theorem.

Let ${F}^t$ be given as in \eqref{eq.p.106}, where $\mu : I \longrightarrow \R$ is a smooth real function satisfying 
\eqref{eq.p.110}, with $\mu(0)=0$. This ordinary differential equation implies that  $(\ref{eq.p.114})$ holds.

\end{proof}

\begin{proof}[{\bf{Proof of Theorem \ref{res.12novo}}}]
(i) If $m=n$, then it follows from Theorem $\ref{res.16}$, that 
\begin{eqnarray}\label{eq.p.283}
{F}^{t}(x)=F(x)+\mu(t)N(x),
\end{eqnarray}
$\forall (x,t) \in  \mathbb{S}^{n}_{r_{0}} \times I$, where, from $(\ref{eq.p.114})$, $\mu(t)$ satisfies
$\prod_{i=1}^{n}(1-k_{i}\mu(t))=e^{t}$.  
Since all the principal curvatures are equal to $k=1/r_{0}$, we have that
$
1-{\mu(t)}/{r_{0}}=e^{t/n}$, 
$\forall t \in I$ and $\forall x \in \mathbb{S}^{n}_{r_{0}}$. 
Thus, since $N(x)=-\frac{1}{r_{0}}F(x)$, from $(\ref{eq.p.283})$ we have
$
{F}^{t}(x)=e^{t/n}F(x)$,
$\forall t \in \R$ and $\forall x \in \mathbb{S}^{n}_{r_{0}}$, i.e., the solution to the IMCF is an eternal solution and  for each $t\in \R$, ${F}^{t}(\mathbb{S}^{n}_{r_{0}})$ is a sphere centered at the origin with radius $r(t)=r_{0}e^{t/n}$ that converges to the origin  when 
$t \rightarrow -\infty$. 

(ii) If $m\neq n$, from Theorem  $\ref{res.16}$, we have that the family of parallel hypersurfaces is given by 
	\begin{eqnarray}\label{eq.p.287}
	{F}^{t}(x,y)=F(x,y)+\mu(t)N(x,y),
	\end{eqnarray}
	$\forall (x,y,t) \in  \mathbb {S}^{m}_{R_0} \times \E^{n-m} \times I$, where $\mu(t)$  satisfies $(\ref{eq.p.114})$, i.e., 
$	\prod_{i=1}^{n}(1-k_{i}\mu(t))=e^{t}$.  
	Since  $m$ principal curvatures are equal to $k \not=0$ and $n-m$ are null principal curvatures, we have 
$	1-k\mu(t)=e^{t/m}$, $\forall t \in I$, 
and hence 
$	\mu(t)=(1-e^{t/m})/k$, 	$\forall t \in \R$. 
Without loss of generality, we consider $F(x,y)=(f(x),y)$,  $\forall (x,y) \in \mathbb {S}^{m}_{r_0} \times \E^{n-m}$, where $f:\mathbb{S}^{m}_{r_{0}} \longrightarrow \E^{m+1}$ is an isometric immersion of the sphere centered at the origin with radius $r_{0}>0$ and dimension $m\neq n$ in  $\E^{m+1}$. Moreover, $N(x,y)=\left(-f(x)/{r_0},0\right)$ is the inward unit normal vector field of $F(\mathbb {S}^{m}_{r_0} \times \E^{n-m})$, $\forall (x,y) \in \mathbb {S}^{m}_{r_0} \times \E^{n-m}$.
Thus, since $k=1/r_{0}$, we conclude from \eqref{eq.p.287} that
\begin{eqnarray}
{F}^{t}(x,y)=(e^{t/m}f(x),y), \nonumber
\end{eqnarray}
$\forall (x,y,t) \in  \mathbb {S}^{m}_{r_0} \times \E^{n-m} \times \R$. Therefore, the solution to the IMCF is eternal. Moreover, ${F}^{t}(\mathbb {S}^{m} \times \E^{n-m})$, $\forall t \in \R$  is a cylinder whose  base  has radius $r(t)=r_{0}e^{t/m}$ that converges to 
 an $(n-m)$-dimensional euclidean subspace  of $\E^{n+1}$.

\end{proof}

In order to prove that the IMCF in the hyperbolic space tends to the conformal boundary 
$\partial\mathbb{H}^{n+1}$ when  $t\rightarrow \infty$, we will need 
 the following lemma.

\begin{lemma}\label{res.23}
Let $F:M^{n} \longrightarrow \mathbb{H}^{n+1}$ be an oriented isoparametric hypersurface with  unit normal vector field $N$ such that the mean curvature $H$ of $F(M)$ satisfies $H(x)>0$, $\forall x \in M$. Let ${F}^t:M^{n}  \longrightarrow \mathbb{H}^{n+1} \subset \mathbb{L}^{n+2}$ be the solution to the IMCF by parallel hypersurfaces with initial condition $F$ defined on the maximal interval $I$. If $(0,+\infty) \subset I$ and 
$$
\lim\limits_{t \longrightarrow +\infty}\cosh(\mu(t))=+\infty,
$$ 
then ${F}^{t}(M)$ converges to the boundary $\partial\mathbb{H}^{n+1}$ when $t \longrightarrow +\infty$.
\end{lemma}

\begin{proof}[{\bf{Proof of Lemma \ref{res.23}}}]
For $x \in M^{n}$, $F(x)$ and  $N(x) \in \mathbb{L}^{n+2}$. We consider the coordinate functions of $F$ and $N$ in $\mathbb{L}^{n+2}$ 
\begin{eqnarray}
F(x)=(F_{1}(x),\dots,F_{n+2}(x)) ~~ \text{  and  } ~~ N(x)=(N_{1}(x),...,N_{n+2}(x)), \nonumber
\end{eqnarray}
where $F_{j}, N_{j}: M^{n} \longrightarrow \R$, $\forall j$ with $1 \leq j \leq n+2$, and $F_{1}(x) \geq 1$. It follows from Theorem  $\ref{res.16}$, that the  hypersurfaces parallel to $F$ are given by 
$
{F}^t(x)=\cosh(\mu(t))F(x)+\sinh(\mu(t))N(x)
$, 
whose  coordinate functions are  
\begin{eqnarray}\label{coord}
{F}^t_{j}(x)=\cosh(\mu(t)) F_{j}(x)+\sinh(\mu(t)) N_{j}(x), \nonumber
\end{eqnarray}
$\forall j$ with $1 \leq j \leq n+2$ and $\forall t \in I$, with ${F}^t_{1}(x) \geq 1$. Moreover, $\forall t \in I$,   
\begin{eqnarray}\label{eq.p.231}
\sum_{j=2}^{n+2} ({F}^t_{j})^{2}=({F}^t_{1})^{2}-1.
\end{eqnarray}

In order to prove that ${F}^t(M)$ converges to the boundary $\partial 
\mathbb{H}^{n+1}$, when $t \longrightarrow +\infty$,  we will use the isometry 
$G: \mathbb{H}^{n+1}\subset \mathbb{L}^{n+2} \longrightarrow  \mathbb{B}^{n+1}\subset \E^{n+1}$, between the hyperbolic space and the Poincar\'e ball model $\mathbb{B}$ defined by
\begin{eqnarray} 
G(x_{1},\ldots,x_{n+2})=\frac{1}{x_{1}+1}(x_{2},\ldots,x_{n},-x_{n+2}),
\end{eqnarray}
where  $\mathbb{B}$ is the ball
 ${B}=\{(x_{1},\dots,x_{n+1}) \in \E^{n+1}\, ; \, \sum_{i=1}^{n+1} x_{i}^{2}<1\}$ with the metric,  for $x\in B$ and $u,v\in T_x {B}$,  
 is given by $
\langle u,v \rangle_{x,B}=
\frac{4}{[1-\sum_{i=1}^{n+1} x_{i}^{2}]^{2}} 
\sum_{i=1}^{n+1} u_iv_i
$.   
We will show  that 
$G({F}^t(x))$ converges to the boundary $\partial \mathbb{B}^{n+1}$, by showing that $|G({F}^t(x))|^2$ converges to 1.

 Restricting $G$ to ${F}^t(M )$, we have 
\begin{eqnarray}\label{functionG}
G({F}^t(x))=\frac{
1}{{F}^t_{1}(x)+1} ({F}^t_{2}(x),...,{F}^t_{n+1}(x),-{F}^t_{n+2}(x)). \nonumber
\end{eqnarray}
Considering the metric of the Poincar\'e ball and in view of \eqref{eq.p.231} we have 
\[
 |G({F}^t(x))|_B^2 = \sum_{j=2}^{n+2}\frac{({F}^t_{j}(x))^2}{( {F}^t_{1}(x))^2+1)^{2}}=\frac{{F}^t_{1}(x)-1}{{F}^t_{1}(x)+1}
= \frac{F_{1}(x)+N_{1}(x)\tanh(\mu(t))-\frac{1}{\cosh(\mu(t))}}{F_{1}(x)+N_{1}(x)\tanh(\mu(t))+\frac{1}{\cosh(\mu(t))}}.
\]
Taking the limit when $t \rightarrow +\infty$, we conclude that    $\lim\limits_{t \rightarrow +\infty} |G({F}^t(x))|^2=1$. 

\end{proof}

\begin{proof}[{\bf{Proof of Theorem \ref{res.11}}}] 
The  hypersurfaces  
$
{F}^t(x)=\cosh(\mu(t))F(x)+\sinh(\mu(t))N(x)
$, 
are parallel to $F$, $\forall (x,t) \in  \mathbb{E}^{n} \times I$,
where $\mu$ satisfies 
$\prod_{i=1}^{n}(\,\cosh(\mu(t))-k_{i}\sinh(\mu(t))\,)=e^{t}$. 
Since all the principal curvatures are equal to $k=\pm 1$, we conclude that 
\begin{eqnarray}\label{eq.p.189}
k\sinh(\mu(t))=\cosh(\mu(t))-e^{t/n},
\end{eqnarray}
$\forall t \in I$.
The square on both sides of $(\ref{eq.p.189})$ and $k=\pm 1$, imply that 
$
\cosh(\mu(t))=\cosh(\pm t/n)$ and $\sinh(\mu(t))=\sinh\left(\mp\frac{t}{n}\right)$, 
$\forall t \in \R$. Hence, $\mu(t)=\mp {t}/{n}$ and 
$\lim\limits_{t \rightarrow +\infty}\cosh(\mu(t))=+\infty$.

It follows from Lemma \ref{res.4} that the principal curvatures ${k}^{t}_{i}$ of ${F}^{t}(\E^{n})$ are given by
$k_i^t=k$,  
and hence ${F}^{t}(\E^{n})$ is a horosphere, $\forall t \in \R$. Considering $\{E_{i}\}_{i=1}^{n}$ a local orthonormal frame of principal directions of $F(\E^{n})$ at  $x \in \E^{n}$, it follows   from Lemma \ref{res.4}, that the first fundamental form at $F^t(x)$  is  given by 
\begin{eqnarray}
I^{t}_{x}(E_{i},E_{j})=\left[\cosh(\mu(t))\mp \sinh(\mu(t))\right]^{2} \delta_{ij}=e^{2t/n}\delta_{ij}, \nonumber
\end{eqnarray}
$\forall t \in \R$. Thus, 
$\lim\limits_{t \rightarrow -\infty}I^{t}_{x}(E_{i},E_{j})=0$. 
Hence  ${F}^{t}
$ is defined $\forall t \in \R$, i.e., the solution to the IMCF  is an eternal solution. Moreover, it follows, from Lemma \ref{res.23} that   
${F}^t(\E^{n})$ converges to a point when $t\rightarrow -\infty$ and it converges to $\partial\mathbb{H}^{n+1}$ when $t \longrightarrow +\infty$.

\end{proof}

\begin{proof}[{\bf{Proof of Theorem \ref{res.22}}}]
We consider the family of  hypersurfaces ${F}^{t}(x)=\cosh(\mu(t))F(x)+\sinh(\mu(t))N(x)$ parallel to $F$. 
 Since all the principal curvatures of $M$ are equal to $k$, we have that  $\mu(t)$ satisfies
\begin{eqnarray}\label{eq.p.196}
k\sinh(\mu(t))=\cosh(\mu(t))-e^{t/n},
\end{eqnarray}
$\forall t \in I$. 
The square of both sides of $(\ref{eq.p.196})$ implies that
\begin{eqnarray}
\left(k^{2}-1\right)\cosh^{2}(\mu(t))+2e^{t/n}\cosh(\mu(t))-\left(k^{2}+e^{2t/n}\right)=0, \nonumber
\end{eqnarray}
$\forall t \in I$. 
By solving this equation  for $\cosh(\mu(t))$ and using the fact that $\mu(0)=0$, we obtain
\begin{eqnarray}\label{eq.p.200}
\cosh(\mu(t))=\frac{e^{t/n}-|k|\sqrt{q(t)}}{1-k^{2}}.
\quad \mbox{ where } \quad q(t)=e^{2t/n}+k^{2}-1. 
\end{eqnarray}
Equations  $(\ref{eq.p.196})$ and $(\ref{eq.p.200})$ imply that
\begin{eqnarray}
\sinh(\mu(t))=sgn(k)\frac{|k|e^{t/n}-\sqrt{q(t)}}{1-k^{2}}. \nonumber
\end{eqnarray}
Moreover
$	\lim\limits_{t \rightarrow +\infty}\cosh(\mu(t))=+\infty$. 

\item $(i)$ If $0<|k|<1$, then  $q(t) \geq 0,$ whenever 
 $t\geq t^*$ where $ t^*=\frac{n}{2}\ln(1-k^{2})$. 
Let $\{E_{j}\}_{j=1}^{n}$ be a local orthonormal frame of principal directions of $F(M)\subset \mathbb{H}^{n+1}$ at $x \in M$. It follows from Lemma $\ref{res.4}$ and $(\ref{eq.p.196})$ that, $\forall t\geq t^{*}$, the fundamnetal forms are given by 
\begin{eqnarray*}
I^{t}_{x}(E_{i},E_{j})=\left[\cosh(\mu(t))-k \sinh(\mu(t))\right]^{2}\delta_{ij}=e^{2t/n}\delta_{ij},
\end{eqnarray*}
\begin{eqnarray}
{II}^{t}_{x}(E_{i},E_{j})=\delta_{ij}sgn(k)e^{t/n} \sqrt{q(t)}. \nonumber
\end{eqnarray}
Hence, at   $t=t^{*}$, we have  
$I^{t^{*}}_{x}(E_{i},E_{j})=(1-k^{2})\delta_{ij}$ and since $q(t^{*})=0$, 
${II}^{t^{*}}_{x}(E_{i},E_{j})=0$. However, the mean curvature of the 
parallel surfaces cannot vanish. Therefore, ${F}^t(M)$ is defined for $t \in (t^{*},+\infty)$, i.e., the solution to the IMCF is an immortal solution,   
that converges to a totally geodesic hypersurface of $\mathbb{H}^{n+1}$ 
when $t\rightarrow t^*$ and it converges to $\partial\mathbb{H}^{n+1}$ when $t \longrightarrow +\infty$, as a consequence of   Lemma \ref{res.23}.  

\item $(ii)$ If $|k|>1$, then $q(t) \geq 0$, $\forall t \in \R$. Therefore, 
${F}^{t}$ is defined for $t \in \R$, i.e., it is an eternal solution.  
Consider $\{E_{j}\}_{j=1}^{n}$ a local orthonormal frame of principal directions of $F(M)\subset \mathbb{H}^{n+1}$ at $x \in M$. 
It follows from \eqref{eq.p.196} and Lemma \ref{res.4},  that
$I^{t}_{x}(E_{i},E_{j})=e^{2t/n}\delta_{ij}$, 
$\forall t \in \R$. Hence, 
$	\lim\limits_{t \rightarrow -\infty}I^{t}_{x}(E_{i},E_{j})=0$. 
 Therefore, ${F}^{t}(M)$ converges to a point when $t \rightarrow -\infty$ and it  converges to the boundary $\partial\mathbb{H}^{n+1}$ 
of the hyperbolic space,  when $t \longrightarrow +\infty$.

\end{proof}

\begin{proof}[{\bf{Proof of Theorem \ref{res.17}}}]
From Theorem $\ref{res.16}$, we have that 
	\begin{eqnarray}
	{F}^{t}(x)=\cosh(\mu(t))F(x)+\sinh(\mu(t))N(x), \nonumber
	\end{eqnarray}
$\forall (x,t) \in  \mathbb{S}^{m}\times \mathbb{H}^{m} \times I$,	where
$\prod_{i=1}^{n}(\,\cosh(\mu(t))-k_{i}\sinh(\mu(t))\,)=e^{t}$. 
Since,  we are assuming that the principal curvatures $k_{1}$ and $k_{2}$ have the same multiplicity $m$,  
using the fact that 
 $k_{1}k_{2}=1$, we have that 
	\begin{eqnarray}\label{eq.p.119}
	\frac{k_{1}+k_{2}}{2}\sinh(2\mu(t))=\cosh(2\mu(t))-e^{t/m},
	\end{eqnarray}
$\forall t \in I$. 
We introduce the notation $a=(k_1+k_2)/2$, where $a>1$ by definition.   
The square on both sides of $(\ref{eq.p.119})$ implies that $\cosh(2\mu(t))$ must satify the algebraic equation 
	\begin{eqnarray}\label{eq.r.2}
	\left((a^2-1\right)\cosh^{2}(2\mu(t))+2e^{t/m}\cosh(2\mu(t))-(a^2+e^{2t/m})=0.
	\end{eqnarray}
Since  $\mu(0)=0$, it follows from \eqref{eq.r.2} and  \eqref{eq.p.119}, that $\forall t \in \R$,
	\begin{eqnarray}\label{eq.r.1}
	\cosh (2\mu(t))=\frac{-e^{t/m}+a\sqrt{q(t)}} {a^2-1}, 
	\qquad 
	\sinh (2	\mu(t))=\frac{-a e^{t/m}+\sqrt{q(t)}} {a^2-1},
	\end{eqnarray}
where 	
\begin{eqnarray*}
    q(t):=e^{2t/m}+a^2-1, 
\end{eqnarray*}
$\forall t \in \R$. The explicit expressions for  $\cosh(\mu(t))$ and $\sinh(\mu(t))$ can be easily obtained from \eqref{eq.r.1}.
Therefore, ${F}^t$ is defined for $t\in \R$, i.e., the solution to the IMCF  is eternal. Moreover,  $k_{1}>1$ implies that 
$	\lim\limits_{t \longrightarrow +\infty}\cosh(\mu(t))=+\infty$.

	Consider an orthonormal frame $\{E_{1},\ldots,E_{m},E_{m+1},...,E_{2m}=E_{n}\}$ of principal directions on $F(\mathbb{S}^{m} \times \mathbb{H}^{m}) \subset \mathbb{H}^{n+1} \subset \mathbb{L}^{n+2}$, in $x \in \mathbb{S}^{m}\times \mathbb{H}^{m}$, such that $E_{1}, ..., E_{m}$ are tangent to $\mathbb{S}^{m}$ and $E_{m+1}, ...,E_{n}$ are tangent to $\mathbb{H}^{m}$. From Lemma $\ref{res.4}$, 
 	the first fundamental form is given by 
 \begin{eqnarray*}
	I^{t}_{x}(E_{i},E_{j})=
\left\{
\begin{array}{lcl}	
	\left(k_{1} \sqrt{q(t)}+\frac{1-k_{1}^{2}}{2}\right)\delta_{ij},  &\quad  \mbox{ if } &  1\leq i,j\leq m, \\
\left(k_{2} \sqrt{q(t)}+\frac{1-k_{2}^{2}}{2}\right)\delta_{ij},  \quad  & \mbox{ if } &  m+1 \leq i,j\leq n.
\end{array}
\right.
	\end{eqnarray*}
Hence, from the expression of $q(t)$,  for $1 \leq i,j \leq m$, and  since  $k_1k_2=1$, we get
	\begin{eqnarray}\label{eq.a.8}
	\lim \limits_{t \longrightarrow -\infty}I^{t}_{x}(E_{i},E_{j})=0, 
	\end{eqnarray}
while for  $m+1 \leq i,j\leq n$,   
\[
\lim \limits_{t \longrightarrow -\infty}I^{t}_{x}(E_{i},E_{j})=(1-k_2^2)\delta_{ij}. 
\]	
 Therefore, ${F}^{t}(\mathbb{S}^{m}\times \mathbb{H}^{m})$ converges to an $m$-dimensional submanifold of $\mathbb{H}^{n+1}$ when $t \rightarrow -\infty$ and  it follows from Lemma \ref{res.23}, that  it converges to the boundary $\partial\mathbb{H}^{n+1}$ when $t \rightarrow +\infty$. 

\end{proof}

Before proving Theorem \ref{resgeral}, we state some well known relations between the principal curvatures of an isoparametric hypersurface of the sphere  that will be used in the proof. For the expressions of the distinct principal curvatures in terms of the first curvature see for example \cite{ReisTenenblat}.  

\begin{remark} \label{relaki}{\rm Consider an $n$-dimensional isoparametric hypersurface of the unit sphere $\mathbb{S}^{n+1}$,  
with $g$ distinct principal curvatures, $k_i$, $i=1,...,g$. Then 

\noindent {\bf 1)} If $g=1$, we are considering $k_1>0$.

\noindent {\bf 2)} If $g=2$, then  
$k_2=-{1}/{k_{1}}$, where w.l.o.g. we may consider $k_1>0$.

\noindent {\bf 3)} If $g=3$, w.l.o.g we may consider $k_{1}>\sqrt{3}$, then 
\[
	k_{2}=\frac{\sqrt{3}k_{1}-3}{3k_{1}+\sqrt{3}}, \qquad  k_{3}=\frac{\sqrt{3}k_{1}+3}{\sqrt{3}-3k_{1}}, \qquad  k_{3}<0<k_{2}<k_{1}.
\]
Thus,  a straightforward computation shows that the following  identities hold:
	\begin{eqnarray}\label{eq.p.44}
\sum_{i=1}^3k_{i}=\frac{3k_{1}(k_{1}^{2}-3)}{3k_{1}^{2}-1},\qquad  
\sum_{\substack {i,j=1 \\ i<j}}^3 k_{i}k_{j}=-3,
\qquad k_{1}k_{2}k_{3}=-\frac{k_{1}(k_{1}^{2}-3)}{3k_{1}^{2}-1}.
	\end{eqnarray}
	
\noindent {\bf 4)} If $g=4$, w.l.o.g. we may consider $k_{1}>1$, hence 	
\[
k_{2}=\frac{k_{1}-1}{k_{1}+1},\qquad 
 k_{3}=-\frac{1}{k_{1}} \qquad 
k_{4}=\frac{k_{1}+1}{1-k_{1}}\quad \mbox{ and } \quad k_{4}<k_{3}<0<k_{2}<k_{1}.
\]
Thus,  
\begin{eqnarray}\label{eq.p.42}
\sum_{i=1}^4 k_i=\frac{k_{1}^{4}-6k_{1}^{2}+1}{k_{1}(k_{1}^{2}-1)}, \quad  
k_{1} k_{3}=k_{2}k_{4}=-1, \quad  (k_{1}+k_{3})(k_{2}+k_{4})=-4, \quad  k_{1}k_{2}k_{3}k_{4}=1. 
\end{eqnarray}

\noindent {\bf 5)} If $g=6$,  w.l.o.g. we may consider $k_{1}>\sqrt{3}$, 	then 
\[
k_{2}=\frac{\sqrt{3}k_{1}-1}{k_{1}+\sqrt{3}}, \qquad 
  k_{3}=\frac{\sqrt{3}k_{1}-3}{3k_{1}+\sqrt{3}}, \qquad  k_{4}=-\frac{1}{k_{1}}, \qquad 
  k_{5}=\frac{\sqrt{3}k_{1}+3}{\sqrt{3}-3k_{1}} \quad \text{ and } \quad  k_{6}=\frac{\sqrt{3}k_{1}+1}{\sqrt{3}-k_{1}}, 
\]  
\vspace*{.1in}

\noindent where $k_{6}<k_{5}<k_{4}<0<k_{3}<k_{2}<k_{1}$. Thus,
\begin{eqnarray}\label{eq.p.500}
\sum_{j=1}^6  k_{j} =3 \frac{k_{1}^{6}-15k_{1}^{4}+15k_{1}^{2}-1}{3k_{1}^{5}-10k_{1}^{3}+3k_{1}}, \qquad k_{1}k_{4}=k_{3}k_{6}=k_{2}k_{5}=k_{1}k_{2}k_{3}k_{4}k_{5}k_{6}=-1.  
\end{eqnarray}

We observe that the identities \eqref{eq.p.44} - \eqref{eq.p.500} are obtained after a straightforward computation, where each principal curvature on the left hand side of the  identity is  substituted by its expression in terms of $k_1$. The identities will be useful in the proof of Theorem \ref{resgeral}.
}
\end{remark}

\vspace{.2in}
 
We will now prove Theorem \ref{resgeral}. 

\begin{proof}[{\bf{Proof of Theorem \ref{resgeral}}}] Let $F$ be an isoparametric hypersurface of the sphere, with $g$  distinct principal curvatures, $k_j$, $j=1...g$  and  $g\in\{1,2,3,4,6\}$.  Without loss of generality we may chose $k_1>0$, as in Remark \ref{relaki}.  

It follows from Theorem $\ref{res.16}$, that the family of hypersurfaces
 parallel to $F$  is a solution to the IMCF if  
${F}^{t}(x)=\cos(\mu(t))F(x)+\sin(\mu(t))N(x)$,
$\forall (x,t) \in M \times I$, $0\in I$, 	where since all the curvatures have the same multiplicty $m$, $\mu(t)$ satifies
\begin{eqnarray} \label{eq.p.443}
\prod_{j=1}^{g}\left( \cos(\mu(t))-k_{j}\sin(\mu(t))\right) =e^{t/m}.
\end{eqnarray}
We introduce the notation 
\[
a=\frac{1}{g} \sum_{j=1}^g k_j=\frac{H}{n}.
\]
Observe that if $g=1$ then $a=k_1$ and $m=n$. Moreover, note that $a>0$. 

A straightforward computation shows that  from 
the relations in Remark \ref{relaki}, and in particular   
\eqref{eq.p.44}-\eqref{eq.p.500}, we have    $\forall t \in I$ and $g\in\{ 1,2,3,4,6\}$,  
\begin{eqnarray}\label{eq.p.171}
\cos(g\mu(t))-a\sin(g \mu(t))=e^{t/m}.
\end{eqnarray}
The square of $(\ref{eq.p.171})$  implies that $\cos(g\mu(t))$ satisfies 
the  algebraic equation
\begin{eqnarray}\label{eq.p.172}
(a^{2}+1)\cos^{2}(g\mu(t))-2e^{t/m}\cos(g\mu(t))+e^{2t/m}-a^{2}=0,
\end{eqnarray}
whose solution, since  $\mu(0)=0$, is given by 
\begin{eqnarray}\label{eq.p.173}
\cos(g\mu(t))=\frac{e^{t/m}+ a\sqrt{q(t)}}{a^{2}+1},
\quad\mbox{ where }\quad  q(t)=a^{2}+1-e^{2t/m}.
\end{eqnarray}
 Note that $q(t) > 0$, whenever $t<t^*$, and $q(t^*)=0$, for 
\begin{equation}\label{tstar}
t^*= \frac{m}{2}\ln(a^{2}+1).
\end{equation} 
Note that $\cos(g\mu(t))>0$, $\forall t\leq  t^{*}$.  Moreover, 
it follows from $(\ref{eq.p.171})$ and $(\ref{eq.p.173})$ that $\forall t \leq  t^{*}$,  
\begin{eqnarray}\label{eq.a.40}
\sin(g\mu(t))=\frac{-ae^{t/m}+\sqrt{q(t)}}{a^{2}+1} \quad  \mbox{ and }
\quad 
\tan(g\mu(t))=\frac{-ae^{t/m}+\sqrt{q(t)}}{e^{t/m}+ a\sqrt{q(t)}}.
\end{eqnarray}

We claim that 
\begin{equation}\label{cosmupos}
\cos(\mu(t))>0, \qquad \mbox{for all} \quad t\leq t^*.
\end{equation}
 In fact, otherwise, since  $\cos(\mu(0))=1$  if $\cos(\mu(t))$ changes  sign  then there exists 
$t_0\leq t^*$ such that $\cos(\mu(t_0))=0$.
 Observe that it follows from the trigonometric identities 
\begin{equation}\label{cosgmucos}
\cos(g\mu(t))=\left\{
\begin{array}{lcl}
\cos(\mu(t)), & \mbox{  if } & g=1,\\
2\cos^2(\mu(t))-1, & \mbox{  if }& g=2,\\
4\cos^{3}(\mu(t))-3\cos(\mu(t)), & \mbox{  if }& g=3,\\
8\cos^{4}(\mu(t))-8\cos^{2}(\mu(t))+1,  & \mbox{  if }& g=4,\\
32\cos^{6}(\mu(t))-48\cos^{4}(\mu(t))+18\cos^{2}(\mu(t))-1,& \mbox{  if } & g=6,   
\end{array}
\right.
\end{equation}
that, for $g=1, 2, 3, 6 $,  such $t_0$ does not exist since $\cos(g\mu(t_0))>0$, $\forall t\leq  t^{*}$. If $g=4$, $\cos(\mu(t_0))=0$ implies that $\cos(4\mu(t_0))=1$ and hence $\sin(4\mu(t_0))=0$. However, in this case, \eqref{eq.p.173} and \eqref{eq.a.40} imply $(a^2+1)(e^{t_0/m}-1)=0$, i.e., $t_0=0$, which contradicts  $\cos(\mu(t_0))=0$. 

Moreover, it follows  from \eqref{cosmupos} and \eqref{eq.p.443} that 
\begin{eqnarray}\label{eq.p.382}
(\cos(\mu(t)))^g\prod_{j=1}^{g} [1-k_{j}\tan(\mu(t))] 
=e^{t/m},
\end{eqnarray}
and therefore, 
\begin{eqnarray}\label{eq.p.383}
\tan(\mu(t))\not=\frac{1}{k_{j}}, \quad j=1,...g, \quad \mbox{ and }\forall t\leq  t^{*}.
\end{eqnarray}

We will now show that the hypersurface $F^t(M)$ is well defined 
for all $t\leq t^*$. Moreover, it converges to a minimal hypersurface of the sphere when $t\rightarrow t^*$ and it converges to a submanifold of 
dimension $n-m$ when $t\rightarrow -\infty$. 
 Let $E_{j,1},\dots,E_{j,m}$, for each $j$, $1\leq j\leq g$, be an orthonormal frame  of principal directions on $F(M) \subset \mathbb{S}^{n+1} \subset \E^{n+2}$. Using Lemma $\ref{res.4}$ one obtains that, $\forall t \leq t^{*}$, the first fundamental form 
\begin{eqnarray}\label{eq.p.180}
I^{t}(E_{j,r}\,,\,E_{j,s})=\cos^{2}(\mu(t))\left[1-k_{j}\tan(\mu(t))\right]^{2}\delta_{rs},
\end{eqnarray}
$\forall r, s$ with $1\leq r,s \leq m$ and $j=1,...,g$ and hence  $I^{t}(E_{j,r}\,,\,E_{j,r})\not=0$, since  \eqref{cosmupos} and \eqref{eq.p.383} hold.


Moreover, from  \eqref{eq.p.173}, we get
\begin{eqnarray}
\lim \limits_{t \longrightarrow -\infty}\cos(g\mu(t))=\frac{a}{\sqrt{a^{2}+1}} \nonumber
\end{eqnarray}
and since $a>0$, from \eqref{eq.p.382} and \eqref{cosgmucos} we conclude that 
\begin{equation}\label{limcosmuneq0}
\lim \limits_{t \longrightarrow -\infty} \cos(\mu(t)) \not= 0.
\end{equation}
It follows  from  $H>0$ and \eqref{eq.p.112}, that $\tan(\mu(t))$ is a decreasing function and hence $\tan(\mu(t))>0$, for all $t<0$. 
Therefore, from \eqref{eq.p.382},  
\eqref{limcosmuneq0}, $k_1>0$ and the order of the principal curvatures  $k_j<k_1$, for all $j\neq 1$ (see Remark \ref{relaki}), we conclude that  
\[ 
\lim _{t\longrightarrow -\infty}\tan(\mu(t))=\frac{1}{k_{1}}.
\]
Consequently, $(\ref{eq.p.180})$ implies  that  $\forall r, s$ with $1\leq r, s \leq m$,
\begin{eqnarray}
\lim _{t\longrightarrow -\infty}I^{t}(E_{1,r}\,,\,E_{1,s})=0 \quad 
\mbox{ and } 
\lim _{t\longrightarrow -\infty}I^{t}(E_{j,r}\,,\,E_{j,s})=c_j\delta_{rs}
\nonumber
\end{eqnarray}
 for each  $j>1$ and some constant  $c_j>0$. We conclude that  $F^t(M)$ converges to an $(n-m)$-dimensional submanifold of the sphere, when 
$t\rightarrow -\infty$. 
 

It follows  from  
Lemma \ref{res.4}, that the principal curvatures of $F^t$ are given by 
\begin{equation}\label{kitki}
k_i^t= \frac{ \tan(\mu(t))+k_i}{1-k_i\tan(\mu(t))}.
\end{equation}
Therefore, a straightforward computation, using Remark \ref{relaki} and   the relations given in 
\eqref{eq.p.44}-\eqref{eq.p.42},  implies that
the mean curvature ${H}^{t}$ of ${F}^{t}$ is given by 
\begin{eqnarray*}
{H}^{t}=m\sum_{i=1}^g k_i^t = \frac{n}{e^{t/m}}cos(g\mu(t)) (\tan(g\mu(t))+a),
\end{eqnarray*}
Hence  ${H}^{t}=0$ if and only if $\tan(g\mu(t))=-a$. Thus, from $(\ref{eq.a.40})$, we have that ${H}^{t}=0$ if and only if 
$q(t)=0$, i.e., $t=t^{*}$. Therefore, the flow collapses into a minimal surface at $t^*$. 
	
If $g=1$ then $k_1=a$ and it follows from \eqref{kitki} that all the principal curvatures at $t^*$ are equal to $k_1^{t^*}=0$.    
Thus, we conclude that  $ {F}^{t}(M)$ expands from a point and converges to a totally geodesic hypersurface of $\mathbb{S}^{n+1}$ when $t \rightarrow t^{*}$.

If $g\in\{ 2,3,4,6\}$,    
since $F^{t^*}(M)$ is a minimal submanifold we have 
\[
\sum_{i=1}^{g} (k_i^{t^*})^2= -2\sum_{\substack {i,j=1\\ i<j}}^g k_i^{t^*}k_j^{t^*}.
\] 
Using \eqref{kitki} and the expressions of $k_i$ in terms of $k_1$ given in Remark \ref{relaki}, a straightforward computation  shows that
\[
\sum_{\substack{i,j=1\\ i<j}}^g k_i^{t^*}k_j^{t^*}=
  \left\{
\begin{array}{lcl}
-1, & \mbox{  if }& g=2,\\
-3, & \mbox{  if }& g=3,\\
-6,  & \mbox{  if }& g=4,\\
-15,& \mbox{  if } & g=6,    
\end{array}
\right.
\] 
i.e.,  $\sum_{\substack{i,j=1\\ i<j}}^g k_i^{t^*}k_j^{t^*}=-g(g-1)/2$. 
Therefore, the square length of the second fundamental form of the minimal submanifold is given by
\[
|A|^2= m\sum_{i=1}^{g} (k_i^{t^*})^2= m g(g-1)=n(g-1), 
\] 
and its scalar curvature is given by $R=n(n-1)-|A|^2$, i.e.,   
\[
R=m(m-1)g^2=n(n-g). 
\]

We conclude that ${F}^{t}(M)$ is defined on $I=(-\infty,t^{*})$, i.e., 
it is an ancient solution that expands from an $m(g-1)$-dimensional submanifold of $\mathbb{S}^{n+1}$ and converges to a minimal hypersurface of $\mathbb{S}^{n+1}$ when $t \longrightarrow t^{*}$.
${F}^{t^*}(M)$ is a totally geodesic submanifold of the sphere when $g=1$, it is a Clifford minimal hypersurface of the sphere, when $g=2$,  and it is a Cartan type minimal submanifold when $g= 3,4, 6$.       

\end{proof}

 \section{Visualizing some solutions}\label{visual}

In this section, we will visualize some examples of solutions to the IMCF by parallel hypersurfaces in space forms.

\noindent \begin{example}
{\normalfont {Consider the parametrized horosphere $F: (0,\pi) \times (0,2\pi) \subset \E^{2} \longrightarrow \mathbb{H}^{3} \subset \mathbb{L}^{4}$ and its nomal
$F(\theta, \varphi)
=\displaystyle\frac{1}{1+\cos\theta}\left(\frac{3}{2}+\cos\theta,\,
\sin\theta\cos\varphi,\,\sin\theta\sin\varphi,\,-\frac{1}{2}-\cos\theta\right)$
and its inward unit normal vector field 
$N(\theta, \varphi)=
\displaystyle\frac{1}{1+\cos\theta}\left(1+\frac{1}{2}\cos\theta,\, \sin\theta \cos \varphi,\,\sin\theta \sin \varphi,\,-1-\frac{3}{2}\cos\theta\right)$. 

Thus, $k_{1}=k_{2}=-1$ are the principal curvatures. Theorem $\ref{res.11}$, shows that 
${F}^{t}=\cosh (t / 2)\, F+\sinh (t / 2)\, N$,  
 $\forall t \in \R$,  is an eternal solution to the IMCF, whose image is a horosphere, that expands from a point on the boundary $\partial\mathbb{H}^{3}$ and converges to the boundary when  $t \longrightarrow +\infty$. In Figure \ref{fig:Sol3}
 one  can see the evolution in  the Poincaré ball model,  by using the isometry $G: \mathbb{H}^{3} \subset \mathbb{L}^{4} \longrightarrow \mathbb{B}^{3} \subset \mathbb{R}^{3}$ given by \eqref{functionG}.
\begin{figure}[!h]
	\centering
	\includegraphics[scale=0.56]{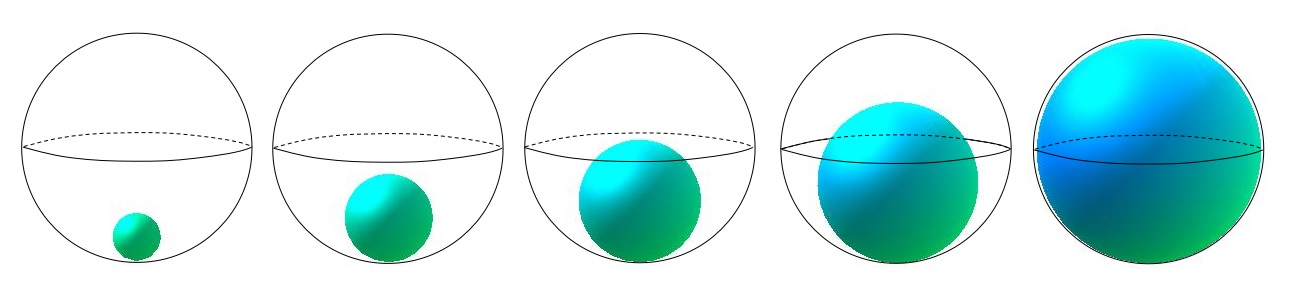}
	\caption{IMCF of the horosphere}
	\label{fig:Sol3}
\end{figure}
}}
\end{example}

\begin{example}
    {\normalfont {Consider the parametrized hyperbolic cylinder 
$F: (0,2\pi) \times \R \longrightarrow \mathbb{H}^{3} \subset \mathbb{L}^{4}$   given by    
$F(\theta, \varphi)=(\sqrt{2} \cosh \varphi, \cos \theta, \sin \theta, \sqrt{2} \sinh \varphi)$ and  
its  
inward 
unit normal vector field 
$N(\theta, \varphi)=(-\cosh \varphi,-\sqrt{2} \cos \theta,-\sqrt{2} \operatorname{sin} \theta,-\operatorname{sinh} \varphi)$. 
Then $k_{1}=\sqrt{2}$ and $k_{2}=\sqrt{2}/2$ are the principal curvatures.   Theorem $\ref{res.17}$, shows that ${F}^{t}=\cosh(\mu(t))F+\sinh(\mu(t))N$, 
	where
 \begin{eqnarray}
\cosh (\mu(t))=\sqrt{-4 e^{t}+\frac{3}{2} \sqrt{1+8 e^{2 t}}+\frac{1}{2}} ~~ \text{  and  } ~~ \sinh(\mu(t))=-sgn(t)\sqrt{-4 e^{t}+\frac{3}{2} \sqrt{1+8 e^{2 t}}-\frac{1}{2}} \nonumber
\end{eqnarray}
is an eternal solution to the IMCF, whose image that expands from a curve of $\mathbb{H}^{3}$ and converges to $\partial\mathbb{H}^{3}$ when $t \longrightarrow +\infty$. In Figure \ref{fig:Sol4}, one can see the evolution in the Poincaré ball model, by using the isometry $G$
given by \eqref{functionG}.
\begin{figure}[!h]
	\centering
	\includegraphics[scale=0.6]{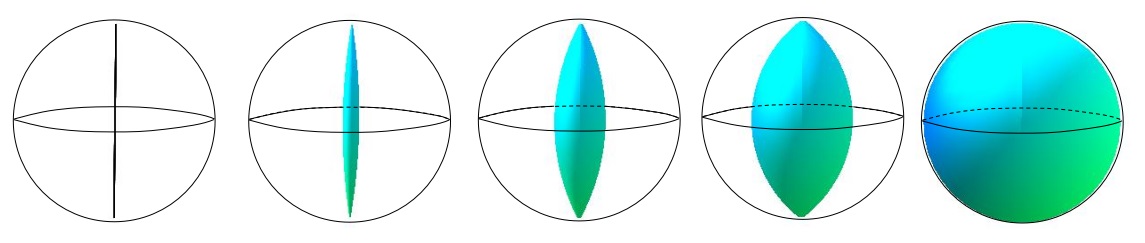}
	\caption{IMCF of the hyperbolic cylinder}
	\label{fig:Sol4}
	\end{figure}
}}
\end{example}

\begin{example}
    {\normalfont {Let 
    $F(\theta, \varphi)=\frac{\sqrt{3}}{3}\left(\sqrt{2} \cos \varphi, \sqrt{2} \sin \varphi,  \cos \theta,  \sin \theta\right). 
$ the parametrized  Hopf Torus in $\mathbb{S}^{3}\subset \mathbb{R}^{4}$ and let $
N(\theta, \varphi)=\frac{\sqrt{3}}{3} \left(\cos \varphi,  \sin \varphi, -\sqrt{2} \cos \theta, -\sqrt{2} \sin \theta\right)$ be 
the inward unit normal vector field. 
Thus, $k_{1}=\sqrt{2}$ and $k_{2}=-\sqrt{2}/2$ are the principal curvatures. Theorem  \ref{resgeral} shows that
$
{F}^{t}=\cos(\mu(t))F+\sin(\mu(t))N$ 
where
\[
\cos (\mu(t))=\sqrt{\frac{8e^{t}+\sqrt{9-8e^{2t}}}{18}+\frac{1}{2}}
\quad 
\mbox{ and }
\sin (\mu(t))=-sgn(t)\sqrt{\frac{1}{2}-\frac{8e^{t}+\sqrt{9-8e^{2t}}}{18}}. \]
is  an ancient solution to the IMCF. 
Its image expands from a curve of $\mathbb{S}^{3}$ and it converges to a minimal hypersurface of $\mathbb{S}^{3}$ when $t \longrightarrow \ln \left(\frac{3\sqrt{2}}{4}\right)$. In Figure \ref{fig:Sol5}, one can see the evolution of its stereographic projection.
\begin{figure}[!h]
	\centering
	\includegraphics[scale=0.44]{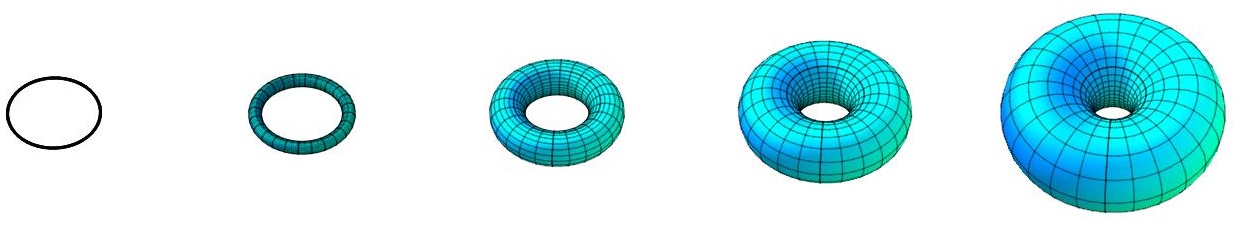}
	\caption{IMCF of the Hopf Torus}
	\label{fig:Sol5}
\end{figure}
}}
\end{example}

\end{document}